\documentclass[a4paper,10pt]{article}

\usepackage[UKenglish]{babel}
\usepackage{amsmath,amscd,amsthm,amssymb}
\usepackage{mathtools}
\usepackage{booktabs}
\usepackage{longtable}
 \usepackage{array}
\usepackage[top=1.25in, bottom=1.25in, left=1.25in, right=1.25in]{geometry} 
\usepackage{booktabs} %for top, middle and bottomline
\usepackage{multirow} %multi column and row spanning
\usepackage{amsmath}
\usepackage{multirow}
\usepackage{tabularx}
\usepackage{multicol}
\usepackage{blindtext}
\usepackage{pbox}
\usepackage{pdflscape}
\usepackage{booktabs,caption,fixltx2e}
\usepackage[flushleft]{threeparttable}

\theoremstyle{plain}
\newtheorem{theorem}{Theorem}[section]
\newtheorem{lemma}[theorem]{Lemma}
\newtheorem{proposition}[theorem]{Proposition}
\newtheorem{corollary}[theorem]{Corollary}
\newtheorem{conjecture}[theorem]{Conjecture}

\theoremstyle{definition}

\theoremstyle{remark}

\renewcommand\footnotemark{}

\DeclareMathOperator{\Irr}{Irr}
\DeclareMathOperator{\tr}{tr}

\DeclareMathOperator{\GL}{GL}
\DeclareMathOperator{\GU}{GU}

\DeclareMathOperator{\G}{\bold{G}}
\DeclareMathOperator{\K}{\bold{K}}
\DeclareMathOperator{\M}{M}
\DeclareMathOperator{\cd}{cd}
\DeclareMathOperator{\A}{A}
\DeclareMathOperator{\Tr}{Tr}
\DeclareMathOperator{\Ind}{Ind}
\DeclareMathOperator{\T}{T}
\DeclareMathOperator{\g}{\mathfrak{g}}
\DeclareMathOperator{\I}{I}
\DeclareMathOperator{\sym}{sym}
\DeclareMathOperator{\diag}{diag}

\begin{document}

%+Title
\title{Enumerating representations of general unitary groups over principal ideal rings of length $2$}
\author{Matthew Levy
\\\\Bielefeld University}
\date{}
\maketitle

\thispagestyle{empty}
%-Title

%+Abstract
\begin{abstract}
We enumerate the number of complex irreducible representations of each degree of general unitary groups of degree $4$ over principal ideal local rings of length two.

\end{abstract}
%-Abstract

%+Contents
%\tableofcontents
%-Contents

\section{Introduction}

Let $F$ be a non-Archimedean local field with ring of integers $\mathfrak{o}$ and let $\mathfrak{p}$ be the unique maximal ideal of $\mathfrak{o}$. Assume that the residue field $\bold{k}=\mathfrak{o}/\mathfrak{p}$ is finite of order $q$ and characteristic $p$. This paper concerns groups of the form $\G=\G(\mathfrak{o})$ where $\G$ is one of the $\mathfrak{o}$-group schemes of type $\A_{3}$, i.e.\ $\GL_4$ or $\GU_4$. Here, the groups $\GU_n(\mathfrak{o})$ are defined over $\mathfrak{o}$ using the non-trivial Galois automorphism of an unramified quadratic extension of $\mathfrak{o}$. For $l\in\mathbb{N}$ we denote by $\mathfrak{o}_l$ the reduction of $\mathfrak{o}_l$ modulo $\mathfrak{p}^l$, i.e.\ $\mathfrak{o}_l=\mathfrak{o}/\mathfrak{p}^l$. We will simply write $\G^{\epsilon}_n$ where $\epsilon\in\{\pm 1\}$ to denote $\G^{1}_n = \GL_n$ and $\G^{-1}_n=\GU_n$. We define the representation zeta function of a group $G$ to be the sum
$$
\zeta_G(s):=\sum_{\chi\in\Irr (G)}\chi(1)^{-s},
$$
where we sum over the complex irreducible characters of $G$, $\Irr (G)$ and $s$ is a complex variable.

The groups $\G^{\epsilon}_n(\mathfrak{o})$ play an important role in the representation theory of the groups $\G^{\epsilon}_n(F)$, being maximal compact subgroups. Furthermore, every continuous representation of $\G^{\epsilon}_n(\mathfrak{o})$ factors through one of the natural homomorphisms $\G^{\epsilon}_n(\mathfrak{o})\rightarrow\G^{\epsilon}_n(\mathfrak{o}_l)$. This brings the study of representations of the groups $\G^{\epsilon}_n(\mathfrak{o}_l)$ to the forefront. The study of representations of groups of type $\A_{n-1}$ has attracted much attention. In 1955 Green \cite{Green} described the characters of the complex irreducible representations of general linear groups over finite fields, i.e.\ groups of the form $\GL_n(\mathbb{F}_q)$. In the 1960s Ennola \cite{Ennola1, Ennola2} gave a description of the characters of the general unitary groups over finite fields and made a curious observation regarding the relationship between characters of general linear and general unitary groups. This `Ennola Duality' is discussed in more detail in Section \ref{ennola}. Recently, Avni, Onn, Klopsch \& Voll \cite{AKOV3} have developed explicit formulae for the representation zeta functions of the groups $\G^{\epsilon}_3(\mathfrak{o}_l)$. Singla \cite{Pooja} has described the representation zeta function of the groups $\GL_4(\mathfrak{o}_2)$, general linear groups over principal ideal rings of length two. In Theorem \ref{main}, the main result of this paper, we give a uniform description of the representation zeta function of the groups $\G^{\epsilon}_4(\mathfrak{o}_2)$ for $\epsilon\in\{\pm 1\}$. We impose no restriction on the residue characteristic of $\mathfrak{o}$. 

\subsection{Ennola Duality}\label{ennola}

In the 1960s Ennola (see \cite{Ennola1, Ennola2}) observed a duality between the character tables of the groups $\GL_n(\mathbb{F}_q)$ and $\GU_n(\mathbb{F}_q)$. In particular, he noted that there exists a finite index set $I=I(n)$ and polynomials $g_i\in\mathbb{Z}[t]$, $i\in I$ such that
$$
\cd(\GL_n(\mathbb{F}_q)) = \{g_i(q):i\in I\}\mbox{ and }\cd(\GU_n(\mathbb{F}_q)) = \{(-1)^{\deg(g_i)}g_i(-q):i\in I\}.
$$
where $\cd(G)=\{\chi(1)\,:\,\chi\in\Irr(G)\}$ denotes the set of character degrees of a group $G$. This phenomenon, known as `Ennola Duality', was later explained by Kawanaka \cite{Kawanaka}. In \cite[Theorem H]{AKOV3}, Avni, Onn, Klopsch \& Voll have observed an anologous form of Ennoloa duality for the groups $\GL_3(\mathfrak{o}_l)$ and $\GU_3(\mathfrak{o}_l)$. In particular, they observed that, for all $g(t)\in\mathbb{Z}[t]$ and $l\in\mathbb{N}$,
$$
g(q)\in\cd(\GL_3(\mathfrak{o}_l))\mbox{ if and only if }(-1)^{\deg g}g(-q)\in\cd(\GU_3(\mathfrak{o}_l)).
$$
For any prime $p$, we write $\cd(G)_{p'}=\{\chi(1)_{p'}\,:\,\Irr(G)\}$ for the prime-to-$p$ parts of the irreducible character degrees of a group $G$. From Theorem \ref{main} we deduce that an analogue of Ennola duality also holds for the groups $\G^{\epsilon}_4(\mathfrak{o}_2)$. More specifically, we observe that the prime-to-$p$ parts of the character degrees satisfy Ennola Duality:

\begin{corollary}\label{corp'}
There are $20$ character degrees ($p'$-part) of the groups $\G^{\epsilon}_4(\mathfrak{o}_2)$:\newline

%\noindent$\cd(\G^{\epsilon}_4(\mathfrak{o}_2))_{p'}= $
%$\{ 1,$
%		$q^2+1$,
%		$(q^2+q\epsilon+1)$,
%		$(q+\epsilon)(q^2+1)$,
%		$(q+\epsilon)^2(q^2+1)$,
%		$(q+\epsilon)(q^2+1)(q^2+\epsilon q+1)$,\\
%
%		\hspace{2cm}$(q+\epsilon)^2(q^2+1)(q^2+\epsilon q+1)$,
%		$(q^2+\epsilon q+1)(q^4-1)$,
%		$(q^2-1)(q^4-1)$,
%		$(q^2+1)(q^2+\epsilon q+1)$,\\
%
%		\hspace{2cm}$(q^2+1)(q^3-\epsilon)$,
%		$(q-\epsilon)(q^2+1)(q^3-\epsilon)$,
%		$(q-\epsilon)(q^2-1)(q^3-\epsilon)$,
%		$(q-\epsilon)(q^3-\epsilon)$,
%		$(q+\epsilon)(q^3+\epsilon)$,\\
%
%		\hspace{2cm}$(q+\epsilon)(q^2+1)(q^3-\epsilon)$,
%		$(q+\epsilon)^2(q^2+1)(q^3-\epsilon)$,
%		$(q^3-\epsilon)(q^4-1)$,
%		$(q+\epsilon)(q^3-\epsilon)(q^4-1)$,\\
%
%		\hspace{2cm}$(q-\epsilon)(q^3-\epsilon)(q^4-1)$,
%		$(q^2-1)(q^3-\epsilon)(q^4-1)$,
%		$(q+\epsilon)(q^2+\epsilon q + 1)(q^4-1)$,\\
%
%		\hspace{2cm}$(q-\epsilon)^2(q^3-\epsilon)$,
%		$(q^2-1)(q^3-\epsilon)\}$.\\

\noindent$\cd(\G^{\epsilon}_4(\mathfrak{o}_2))_{p'}= \cd(\G^{\epsilon}_4(\mathfrak{o}_1))_{p'}\,\cup\,$
$\{ $%1,$
		%$q^2+1$,
		%$(q^2+q\epsilon+1)$,
		$(q+\epsilon)(q^2+1)$,
		%$(q+\epsilon)^2(q^2+1)$,
		%$(q+\epsilon)(q^2+1)(q^2+\epsilon q+1)$,
		%\hspace{2cm}$(q+\epsilon)^2(q^2+1)(q^2+\epsilon q+1)$,
		%$(q^2+\epsilon q+1)(q^4-1)$,
		%$(q^2-1)(q^4-1)$,
		%$(q^2+1)(q^2+\epsilon q+1)$,\\
		%\hspace{2cm}$(q^2+1)(q^3-\epsilon)$,
		%$(q-\epsilon)(q^2+1)(q^3-\epsilon)$,
		%$(q-\epsilon)(q^2-1)(q^3-\epsilon)$,
		%$(q-\epsilon)(q^3-\epsilon)$,\\
		%\hspace{2cm}$(q+\epsilon)(q^2+1)(q^3-\epsilon)$,
		%$(q+\epsilon)^2(q^2+1)(q^3-\epsilon)$,
		$(q^3-\epsilon)(q^4-1)$,
		$(q+\epsilon)(q^3-\epsilon)(q^4-1)$,\\

		\hspace{2cm}$(q-\epsilon)(q^3-\epsilon)(q^4-1)$,
		$(q^2-1)(q^3-\epsilon)(q^4-1)$,
		$(q+\epsilon)(q^2+\epsilon q + 1)(q^4-1)$,\\

		%\hspace{2cm}$(q-\epsilon)^2(q^3-\epsilon)$,
		%$(q^2-1)(q^3-\epsilon)
		\hspace{2cm}$(q^3+\epsilon q^2+q+\epsilon)^2\}$\\
where\\

\noindent$\cd(\G^{\epsilon}_4(\mathfrak{o}_1))_{p'}= $
$\{ 1,$
		$q^2+1$,
		$q^2+\epsilon q+1$,
		%$(q+\epsilon)(q^2+1)$,
		$(q+\epsilon)^2(q^2+1)$,
		%$(q+\epsilon)(q^2+1)(q^2+\epsilon q+1)$,\\
		$(q+\epsilon)^2(q^2+1)(q^2+\epsilon q+1)$,
		$(q^2+\epsilon q+1)(q^4-1)$,\\

		\hspace{2cm}$(q^2-1)(q^4-1)$,
		$(q^2+1)(q^2+\epsilon q+1)$,
		$(q+\epsilon)(q^3+\epsilon)$,
		$(q^2+1)(q^3-\epsilon)$,\\

		\hspace{2cm}$(q-\epsilon)(q^2+1)(q^3-\epsilon)$,
		$(q-\epsilon)(q^2-1)(q^3-\epsilon),$
		$(q-\epsilon)(q^3-\epsilon)\}$.\\

		%\hspace{2cm}$(q+\epsilon)(q^2+1)(q^3-\epsilon)$,
		%$(q+\epsilon)^2(q^2+1)(q^3-\epsilon)$,
		%$(q^3-\epsilon)(q^4-1)$,
		%$(q+\epsilon)(q^3-\epsilon)(q^4-1)$,\\

		%\hspace{2cm}$(q-\epsilon)(q^3-\epsilon)(q^4-1)$,
		%$(q^2-1)(q^3-\epsilon)(q^4-1)$,
		%$(q+\epsilon)(q^2+\epsilon q + 1)(q^4-1)$,\\

		%\hspace{2cm}$(q-\epsilon)^2(q^3-\epsilon)$,
		%$(q^2-1)(q^3-\epsilon)

%\{1,
%		q^2+1,
%		(q^2+q\epsilon+1),
%		(q+\epsilon)(q^2+1),
%		(q+\epsilon)^2(q^2+1),
%		(q+\epsilon)(q^2+1)(q+\epsilon q+1),$$
%		$$(q+\epsilon)^2(q^2+1)(q+\epsilon q+1),
%		(q^2+\epsilon q+1)(q^4-1),
%		(q^2-1)(q^4-1),
%		(q^2+1)(q^2+\epsilon q+1),
%		(q^2+1)(q^3-\epsilon),
%		(q-\epsilon)(q^2+1)(q^3-\epsilon),
%		(q-\epsilon)(q^2-1)(q^3-\epsilon),
%		(q-\epsilon)(q^3-\epsilon),
%		(q+\epsilon)(q^2+1)(q^3-\epsilon),
%		(q+\epsilon)^2(q^2+1)(q^3-\epsilon),
%		(q^3-\epsilon)(q^4-1),
%		(q+\epsilon)(q^3-\epsilon)(q^4-1),
%		(q-\epsilon)(q^3-\epsilon)(q^4-1),
%		(q^2-1)(q^3-\epsilon)(q^4-1),
%		(q+\epsilon)(q^2+\epsilon q + 1)(q^4-1),
%		(q-\epsilon)^2(q^3-\epsilon),
%		(q^2-1)(q^3-\epsilon)
%\}.$$

\noindent Moreover, for $l=1,2$,
$$
g(q)\in\cd(\GL_4(\mathfrak{o}_l))_{p'}\iff (-1)^{\deg(g)}g(-q)\in\cd(\GU_4(\mathfrak{o}_l))_{p'}.
$$
\end{corollary}

%\begin{remark}
In \cite[Theorem H]{AKOV3} the authors note that for $n=3$ there is a case distinction between $l=1$ and $l\geq 2$ for the prime-to-$p$ parts of the set of character degrees, $\cd(\G^{\epsilon}_3(\mathfrak{o}_l))_{p'}$. Corollary \ref{corp'} shows that there is also a case distinction for the prime-to-$p$ parts of set of character degrees between $l=1$ and $l=2$ for $n=4$ but we are unable to say anything for $l\geq 3$. We make the following more general conjecture:
%\end{remark}

\begin{conjecture}[Ennola Duality over $\mathfrak{o}_l$]
For all $l\in\mathbb{N}$
$$
g(q)\in\cd(\GL_n(\mathfrak{o}_l))_{p'}\iff (-1)^{\deg(g)}g(-q)\in\cd(\GU_n(\mathfrak{o}_l))_{p'}.
$$
\end{conjecture}

\subsection{Symmetric matrices and the Frobenius-Schur indicator}

In \cite[Remark 1.3]{AKOV3} the authors note that the special value of the zeta function $\zeta_{\GU_3(\mathfrak{o}_l)}(s)$ at $s = -1$ (i.e.\ the sum of character degrees) is equal to the number of symmetric matrices in $\GU_3(\mathfrak{o}_l)$, that is
$$
\zeta_{\GU_3(\mathfrak{o}_l)}(-1)  = (1+q^{-1})(1+q^{-3})q^{6l} = \mbox{number of symmetric matrices in $\GU_3(\mathfrak{o}_l)$}.
$$
The corresponding assertion for $\GL_3(\mathfrak{o}_l)$ holds only for $l=1$. This observation, that the sum of character degrees is equal to the number of symmetric matrices, was observed by Gow \& Klyachko for the groups $\GL_n(\mathbb{F}_q)$ and Thiem \& Vinroot for the groups $\GU_n(\mathbb{F}_q)$ (see \cite{ThiemVinroot}). This phenomenon fails in the cases $\G^{\epsilon}_4(\mathfrak{o}_2)$. Let $\sym^{\epsilon}_n(\mathfrak{o}_l)$ denote the number of symmetric matrices in $\G^{\epsilon}_n(\mathfrak{o}_l)$. For $n=4$ we have
$$
\sym^{\epsilon}_4(\mathfrak{o}_l)=(1-\epsilon q^{-1})(1-\epsilon q^{-3})q^{20}.
$$ 
We have the following corollary of Theorem \ref{main}
\begin{corollary}\label{cormain}
The representation zeta functions of the groups $\G^{\epsilon}_4(\mathfrak{o}_2)$ evaluated at $s=-1$ are given by
\begin{eqnarray*}
\zeta_{\GU_4(\mathfrak{o}_2)}(-1) &=& q^2(q^2-q+1)(q^{14}+q^7-2q^6-q^5+2q^4-q^3+2q^2+q-2)(q+1)^2;\\
\zeta_{\GL_4(\mathfrak{o}_2)}(-1) &=& q(q^2+q+1)(q^{15}+2q^{10}-2q^8+2q^6-2q^4-4q^2+4)(q-1)^2
\end{eqnarray*}
where $q$ is the cardinality of $\mathfrak{o}_1$. Moreover,
\begin{eqnarray*}
\zeta_{\GU_4(\mathfrak{o}_2)}(-1)-\sym^{-1}_4(\mathfrak{o}_2) &=& q^2(q-2)(q^2+1)(q^2-q+1)(q-1)^2(q+1)^4;\\
\zeta_{\GL_4(\mathfrak{o}_2)}(-1)-\sym^{1}_4(\mathfrak{o}_2) &=& 2q(q^2+q+1)(q^4+2)(q^2+1)(q+1)^2(q-1)^4;\\
\frac{\zeta_{\GU_4(\mathfrak{o}_2)}(-1)}{\sym^{-1}_4(\mathfrak{o}_2)} &= &\frac{q^{14}+q^7-2q^6-q^5+2q^4-q^3+2q^2+q-2}{q^{14}};\\
\frac{\zeta_{\GL_4(\mathfrak{o}_2)}(-1)}{\sym^{1}_4(\mathfrak{o}_2)}&=&\frac{q^{15}+2q^{10}-2q^8+2q^6-2q^4-4q^2+4}{q^{15}}.
\end{eqnarray*}
In particular,
$$
\lim_{q\rightarrow\infty}\frac{\G^{\epsilon}_4(\mathfrak{o}_2)}{\sym^{\epsilon}_4(\mathfrak{o}_2)}=1.
$$
\end{corollary}
The reason for the failure of the sum of character degrees to equal the number of symmetric matrices can be expressed in terms of word maps with automorphisms and generalized Frobenius-Schur indicators. It is known, see \cite{BumpGinzburg}, that for a finite group $G$ with automorphism $\tau$ of order $2$ we have, for each $g\in G$:
\begin{eqnarray}\label{fs}
\sum_{\chi\in\Irr(G)}\mathfrak{i}_{\tau}(\chi)\chi(g) = |\{h\in G:h^{\tau}h = g\}|,
\end{eqnarray}
where $\mathfrak{i}_{\tau}(\chi) = \frac{1}{|G|}\sum_{g\in G}\chi(g^{\tau}g)$ (analogous to the classic Frobenius-Schur indicator when $\tau$ is trivial). We can think of the right hand side of (\ref{fs}) as the number of solutions to the word map given by $h^{\tau}h = g$ where we solve for $h\in G$ for a given element $g\in G$. When $G = \G^{\epsilon}_n(\mathfrak{o}_l)$, $g=1$ and $\tau$ is the transpose-inverse automorphism equation (\ref{fs}) becomes
$$
\sum_{\chi\in\Irr(G)}\mathfrak{i}_{\tau}(\chi)\chi(1) = |\{h\in G:h\mbox{ is symmetric}\}|,
$$
a weighted sum of character degrees. We conclude from Corollary \ref{cormain} that, for $\G^{\epsilon}_4(\mathfrak{o}_2)$, some representations have non-trivial generalized Frobenius-Schur indicator. This contrasts with the field case where all groups $\G^{\epsilon}_n(\mathbb{F}_q)$ have generalised Frobenius-Schur indicators equal to $1$. It would be interesting to see to what extent the final statement of Corollary \ref{cormain} holds for the groups $\G^{\epsilon}_n(\mathfrak{o}_l)$ with $n\geq 4$ and $l\geq 2$.

\section{The representation zeta function of $\G^{\epsilon}_n(\mathfrak{o}_l)$}

Let $F$ be a non-Archimedean local field with ring of integers $\mathfrak{o}$ and let $\mathfrak{p}$ be the unique maximal ideal of $\mathfrak{o}$. Assume that the residue field $\bold{k}=\mathfrak{o}/\mathfrak{p}$ is finite of order $q$ and characteristic $p$. We also fix a uniformiser $\pi$ of $\mathfrak{o}$. A typical example of such a field $F$ is $\mathbb{Q}_p$ (the $p$-adic numbers) with ring of integers $\mathbb{Z}_p$ (the $p$-adic integers), unique maximal ideal $p\mathbb{Z}_p$ and residue field $\mathbb{F}_p$. Let $\mathfrak{O}$ be an unramified quadratic extension of $\mathfrak{o}$, with valuation ideal $\mathfrak{P}$ and residue field $\bold{k}_2$, a quadratic extension of $\bold{k}$. Then $\mathfrak{O}=\mathfrak{o}[\delta]$, where $\delta=\sqrt{\rho}$ for an element $\rho\in\mathfrak{o}$ whose reduction modulo $\mathfrak{p}$ is a non-square in $\bold{k}$, and $\mathfrak{P}=\pi\mathfrak{O}$. Let $\mathfrak{I}$ denote the integral closure of $\mathfrak{O}$ in some fixed algebraic closure of its fraction field, and choose an $\mathfrak{o}$-automorphism $\circ$ of $\mathfrak{I}$ restricting to the non-trivial Galois automorphism of the quadratic extension $\mathfrak{O} | \mathfrak{o}$.  Let $n\in\mathbb{N}$. For a matrix $A = (a_{ij})\in\M_n(\mathfrak{O})$ write $A^{\circ} = ((a_{ij}^{\circ})^{\tr})$, for the conjugate transpose. A matrix is \textit{hermitian} if $A^\circ=A$ and \textit{anti-hermitian} if $A^\circ = -A$. The \textit{standard unitary group} over $\mathfrak{o}$ is the group
$$
\GU_n(\mathfrak{o}) = \{A\in\GL_n(\mathfrak{O}):A^\circ A = \I_n\}.
$$
We also define the corresponding \textit{standard unitary $\mathfrak{o}$-Lie lattice} to be
$$
\mathfrak{gu}_n(\mathfrak{o}) = \{A\in\mathfrak{gl}_n(\mathfrak{O}):A^\circ+A=0\}.
$$
For $l\in\mathbb{N}$ we denote by $\mathfrak{o}_l$ the reduction of $\mathfrak{o}_l$ modulo $\mathfrak{p}^l$, i.e.\ $\mathfrak{o}_l=\mathfrak{o}/\mathfrak{p}^l$, and analogously $\mathfrak{O}_l=\mathfrak{D}/\mathfrak{P}^l$. A matrix $A\in\GU_n(\mathfrak{o}_l)$ is called \textit{hermitian}, respectively \textit{anti-hermitian}, if it is the image of a hermitian, respectively anti-hermitian matrix, modulo $\mathfrak{P}^l$. Recall that we will simply write $\G^{\epsilon}_n$ where $\epsilon\in\{\pm 1\}$ to denote $\G^{1}_n = \GL_n$ and $\G^{-1}_n=\GU_n$.

Our overall aim, reached in Theorem \ref{main}, is to compute a uniform formula for the representation zeta function of the groups of the form $\G^{\epsilon}_4(\mathfrak{o}_2)$, $\epsilon\in\{\pm 1\}$. First we consider general $n\in\mathbb{N}$, specialising to $n=4$ later on. Now we describe a bijection between the irreducible representations of the groups $\G^{\epsilon}_n(\mathfrak{o}_2)$ and the union of the irreducible representations of centralisers of certain matrices over the field $\mathfrak{o}_1$. Details can be found in \cite{Pooja2}. Let $\K^{\epsilon}_n$ denote the kernel of the map $\kappa:\G^{\epsilon}_n(\mathfrak{o}_2)\rightarrow\G^{\epsilon}_n(\mathfrak{o}_1)$. Note that $\K^{\epsilon}_n$ is a finite abelian group and let $\widehat{\K^{\epsilon}_n}$ denote the set of characters of $\K^{\epsilon}_n$. The group $\G^{\epsilon}_n(\mathfrak{o}_2)$ acts on $\widehat{\K^{\epsilon}_n}$ by conjugation: if $g\in\G^{\epsilon}_n(\mathfrak{o}_2)$ and $\phi\in\widehat{\K^{\epsilon}_n}$ then $\phi^g(x) = \phi(x^g)$ for $x\in\K^{\epsilon}_n$. For any $\phi\in\widehat{\K^{\epsilon}_n}$, write $\T^{\epsilon}_n(\phi)=\{g\in\G^{\epsilon}_n(\mathfrak{o}_2):\phi^g=\phi\}$. Let $\mathfrak{C}^{\epsilon}_n$ denote the set of $\G^{\epsilon}_n(\mathfrak{o}_2)$-orbits in $\widehat{\K^{\epsilon}_n}$. It is shown in \cite{Pooja2} that for a character $\phi\in\widehat{\K^{\epsilon}_n}$ there exists a canonical extension $\chi_{\phi}$ to $T^{\epsilon}_n(\phi)$ so that $\chi_{\phi}|_{\K^{\epsilon}_n}=\phi$. By Clifford Theory and \cite{Pooja2} there exists a bijetion between the sets
$$
\amalg_{\phi\in\mathfrak{C}^{\epsilon}_n}\{\Irr(\T^{\epsilon}_n(\phi)/\K^{\epsilon}_n)\} \longleftrightarrow\Irr(\G^{\epsilon}_n(\mathfrak{o}_2))
$$
given by
$$
\delta\mapsto\Ind_{\T^{\epsilon}_n(\phi)}^{\G^{\epsilon}_n(\mathfrak{o}_2)}(\chi_{\phi}\otimes\delta).
$$
%where $\tilde{\delta}$ is a representation of $T^{\epsilon}_n(\phi)$ obtained by composing $\delta$ with the natural projection $T^{\epsilon}_n(\psi)\rightarrow T^{\epsilon}_n(\psi)/\K^{\epsilon}_n$. 
Fix a non-trivial additive character $\psi:\mathfrak{o}_1\rightarrow\mathbb{C}^*$. Define $\g^1_n(\mathfrak{o}_1):=\mathfrak{gl}_n(\mathfrak{o}_1)$. For each matrix $A\in$ $\g^1_n(\mathfrak{o}_1)$ define the character $\psi_A:\K_n^{1}\rightarrow\mathbb{C}^*$ by
$$
\psi_A(\I_n + \pi X) = \psi(\Tr(AX)).
$$
The assignment $A\mapsto\psi_A$ defines an isomorphism $\g^1_n(\mathfrak{o}_1)\cong\widehat{\K_n^{1}}$. Define $\g^{-1}_n(\mathfrak{o}_1)$ to be the subgroup of $\g^1_n(\mathfrak{D}_1)$ such that $X\mapsto I_n + \pi X$ defines an isomorphism $\g^{-1}_n(\mathfrak{o}_1)\cong\widehat{\K_n^{-1}}$. Thus $\g^{-1}_n(\mathfrak{o}_1) = \{X\in\g^1_n(\mathfrak{O}_1):X+X^\circ = 0\}=\mathfrak{gu}_n(\mathfrak{o}_1)$. Let $\mathfrak{S}^{\epsilon}_n$ denote the set of orbits of $\g^{\epsilon}_n(\mathfrak{o}_1)$ by $\G^{\epsilon}_n(\mathfrak{o}_2)$. Since $T^{\epsilon}_n(\phi)/\K^{\epsilon}_n\cong Z_{\G^{\epsilon}_n(\mathfrak{o}_1)}(A)$ we have the following bijection
\begin{eqnarray}\label{eqnsim}
\amalg_{A\in\mathfrak{S}^{\epsilon}_n}\{\Irr(Z_{\G^{\epsilon}_n(\mathfrak{o}_1)}(A))\} \longleftrightarrow\Irr(\G^{\epsilon}_n(\mathfrak{o}_2)).
\end{eqnarray}
It follows that to find the complex irreducible representations of the groups $\G^{\epsilon}_n(\mathfrak{o}_2)$ one can simply write down representatives for the similarity classes $\mathfrak{S}^{\epsilon}_n$ and induce representations of their centralisers. The case $\epsilon = 1$, $n=4$ has been described by Singla \cite{Pooja}. In Theorem \ref{main} we give formulae for the number of irreducible representations of each degree of the group $\GU_4(\mathfrak{o}_2)$. This is achieved by writing down representatives $A$ of similarity classes in $\g^{-1}_4(\mathfrak{o}_1)$ and studying the representations of their centralisers $Z_{\G^{\epsilon}_4(\mathfrak{o}_1)}(A)$. 
%The similarity classes $\mathfrak{S}^{-1}_n$ are described by Ennola in \cite{Ennola1}. The similarity classes $\mathfrak{S}^{-1}_4$ can be obtained from the list of similarity classes $\mathfrak{S}^{1}_4$ which in turn can be found in \cite{Pooja}. 

Let $f(t) = t^d-a_{d-1}t^{d-1}-...-a_0$ be a polynomial over a field $\mathbb{F}_q$ of degree $d$. Define matrices
$$
U(f) = U_1(f) = \begin{pmatrix} 0 & 1\\ 0 & 0 & 1\\ . & . & . & . & . \\ 0 & 0 & 0 & ... &1\\ a_0 & a_1 & a_2 & ... & a_{d-1}\end{pmatrix}
$$
and
$$
U_m(f) = \begin{pmatrix} U(f) & \I_d  \\  & U(f) & \I_d\\ . & . & . & . & \\  &  &  & U(f)\end{pmatrix}
$$
with $m$ diagonal blocks $U(f)$ and where $\I_d$ is the $d\times d$ identity matrix. For a partition $\lambda = \{l_1,l_2,...,l_p\}$ of a positive integer $k$ with $l_1\geq l_2\geq...\geq l_p>0$ write
$$
U_{\lambda}(f) = \begin{pmatrix} U_{l_1}(f) \\  & U_{l_1}(f) \\ . & . & . & . & \\  &  &  & U_{l_p}(f)\end{pmatrix}.
$$
In \cite{Green} Green shows that there is a one-to-one correspondence between similarity clases in $\g^{1}_n(\mathfrak{o}_1)$ and collections of irreducible polynomials with associated partitions satisfying certain conditions. 
%More specifically, a similarity class $C$ in $\g^{1}_n(\mathfrak{o}_1)$ can be represented by the symbol
%$$
%\{f_1^{\nu_C(f_1)},...,f_k^{\nu_C(f_k)}\}
%$$
%where, for each $i$, the $f_i$ are irreducible polynomials appearing in the characteristic polynomial of $C$ and $\nu_C(f_i)$ is the partition associated with $f_i$ in the Jordan canonical form of $C$. 
More specifically, suppose that the characteristic polynomial of a similarity class $C$ is $f_1^{k_1}...f_N^{k_N}$ where the $f_i$ are distinct irreducible polynomials over $\mathfrak{o}_1$, $k_i\geq 0$ and if the respective degrees of the $f_i$ are $d_i$ then $\sum_{i=1}^{N}k_id_i=n$. Then $C$ is similar to the diagonal block matrix
$$
\diag\{U_{\nu_C(f_1)}(f_1), U_{\nu_C(f_2)}(f_2),...,U_{\nu_C(f_N)}(f_N)\}
$$
where each $\nu_C(f_i)$ is a certain partition of $k_i$ depending on $C$. We may therefore represent the similarity class $C$ by the symbol
$$
\{f_1^{\nu_C(f_1)},...,f_N^{\nu_C(f_N)}\}.
$$
Now let $C= \{...,f_i^{\nu_C(f_i)},...\}$ for some irreducible polynomials $f_i$. For  a natural number $d\geq 1$ and a partition $\nu$ other than $0$ write $r_C(d,\nu)$ for the number of irreducible polynomials of degree $d$ appearing in the characteristic polynomial of $C$ with partition $\nu_C(f)=\nu$.  Let $\rho_C(\nu)$ be the partition
$$
(n^{r_C(n,\nu)},(n-1)^{r_C(n-1,\nu)},...).
$$
We say that two similarity classes, $A$ and $B$, are of the same \textit{type} if and only if $\rho_A(\nu)=\rho_B(\nu)$ for every non-zero partition $\nu$. We will also say that two matrices are of the same type if their respective similarity classes are of the same type.
Now let $\rho_\nu$ be a partition valued function on the non-zero partitions $\nu$ (we allow $\rho_\nu$ to take the value zero). The function $\rho_\nu$ describes a type in $\g^{1}_n(\mathfrak{o}_1)$ if and only if
\begin{eqnarray}\label{eqntype}
\sum_{\nu}|\rho_\nu||\nu| = n.
\end{eqnarray}
We analogously define the \textit{type} of a similarity class in $\g^{-1}_n(\mathfrak{o}_1)$. The following lemmas, from \cite{AKOV3}, allow us to describe the types of matrices that occur in $\g^{\epsilon}_n(\mathfrak{o}_1)$:
\begin{lemma}[Lemma 3.2 \cite{AKOV3}]
Let $A,B\in\g^{-1}_n(\mathfrak{o})$ be similar, i.e.\ $\GL_n(\mathfrak{O})$-conjugate. Then $A,B$ are already $\GU_n(\mathfrak{o})$-conjugate.
\end{lemma}
\begin{lemma}[Lemma 3.5 \cite{AKOV3}]
Let $A\in\g^{1}_n(\mathfrak{O}_l)$ with characteristic polynomial $f_A=t^n+\sum_{i=0}^{n-1}c_it^i\in\mathfrak{O}_l[t]$. If $A$ is $\GL_n(\mathfrak{O})$-conjugate to an anti-hermitian matrix, then $c_i^{\circ}=(-1)^{n-i}c_i$ for $0\leq i<n$. %Conversely, if $A$ is cyclic then the latter condition on the coefficients implies that $A$ is $\GL_n(\mathfrak{O})$-conjugate to an anti-hermitian matrix.
\end{lemma}
The following lemma highlights the importance of type in the context of computing representation zeta functions:
\begin{lemma}
If matrices $A$ and $B$ in $\g^\epsilon_n(\mathfrak{o}_1)$ are of the same type, then their centralisers are isomorphic.
\end{lemma}
\begin{proof}
Since $A$ and $B$ are of the same type there exist irreducible polynomials $f_1,...,f_N$ and $g_1,...,g_N$ such that $\deg(f_i)=\deg(g_i)=d_i$ and positive integers $k_i$ such that the characteristic polynomial of $A$ is $f_1^{k_1}...f_N^{k_N}$and the characteristic polynomial of $B$ is $g_1^{k_1}...g_N^{k_N}$. Moreover, there exist partitions $\nu_i$ of the $k_i$ such that $A$ is similar to 
$$
\diag\{U_{\nu_1}(f_1),...,U_{\nu_N}(f_N)\}
$$
and $B$ is similar to
$$
\diag\{U_{\nu_1}(g_1),...,U_{\nu_N}(g_N)\}.
$$
From this it is clear that $A$ and $B$ have isomorphic centralisers.
\end{proof}

For any group $\G^{\epsilon}_n(\mathfrak{o}_1)$, using equation (\ref{eqntype}) we may wrtie down a complete and irredundant list of representatives of types that are characterised by sets of irreducible polynomials and their associated partitions. This list will not depend on the underlying field, $\mathfrak{o}_1$, however each of these representatives may be parameterised by coefficients that do depend on the underlying field. Let $\mathbb{T}^{\epsilon}_n$ denote the set of representatives of types in $\g^{\epsilon}_n(\mathfrak{o}_1)$ and for each $A\in\mathbb{T}^{\epsilon}_n$, let $n_A$ be the total number of similarity classes of type $A$. We will also write $Z_{\G^{\epsilon}_n(\mathfrak{o}_1)}(A)$ for the centraliser of a matrix of type $A$. Then
\begin{eqnarray}\label{eqnzeta}
\zeta_{\G^{\epsilon}_n(\mathfrak{o}_2)}(s)=\sum_{A\in\mathbb{T}^{\epsilon}_n}n_{A}\zeta_{Z_{\G^{\epsilon}_n(\mathfrak{o}_1)}(A)}(s)|\G^{\epsilon}_n(\mathfrak{o}_1):Z_{\G^{\epsilon}_n(\mathfrak{o}_1)}|^{-s}.
\end{eqnarray}

\subsection{Representations of $\G^{\epsilon}_2(\mathfrak{o}_l)$}

Before proceeding with the proof of Theorem \ref{main} we summarise what is already known about the number of complex irreducible representations of each degree of the groups $\G^{\epsilon}_2(\mathfrak{o}_l)$ for $\epsilon\in\{\pm 1\}$ and $l = 1, 2$. For $(\epsilon, n, l) = (1,2,1)$ see Steinberg \cite{Steinberg2}. For $(\epsilon, n, l) = (-1,2,1)$ see Ennola \cite{Ennola2}. The number of irreducible representations of the groups $\G^{\epsilon}_2(\mathfrak{o}_1)$ is given in Table \ref{table21}. Details on the irreducible representations for $(\epsilon, n, l) = (1,2,2)$ are described by Nagornyi \cite{Nag} and Onn \cite{Onn}. For $(\epsilon, n, l) = (-1,2,2)$ see \cite{AKOV3}. Table \ref{table22} is a complete and irredundant list of representatives of similarity class types of $\g^{\epsilon}_2(\mathfrak{o}_1)$ under the action by $\G^{\epsilon}_2(\mathfrak{o}_2)$. Using equation (\ref{eqnzeta}) one can obtain the representation zeta function of the groups $\G^{\epsilon}_2(\mathfrak{o}_2)$.

\begin{table}[h!]
\small
\centering
\caption{Representations of $\G^{\epsilon}_2(\mathfrak{o}_1)$}
  \begin{center}    
    \label{table21}
    \begin{tabular}{| c | c |}%{|*{2}{>{\centering\arraybackslash}p{4cm}|}}% {| c | c | c | c | }
      \hline
      Number of irreducible representations & Degree \\
      \hline
 &\\
      %$\left(\begin{matrix}
               %\rho^a & 0 \\ 0 & \rho^a
               %\end{matrix} \right)$  &
		$q-\epsilon$ & 
		$1$\\
& \\
      %$\left(\begin{matrix}
               %\rho^a & 0\\ 1 & \rho^a
              % \end{matrix} \right)$  & 
		$q-\epsilon$ & 
		$q$\\
&\\
      %$\left(\begin{matrix}
               %\rho^a & 0 \\ 0 & \rho^b
               %\end{matrix} \right)$  & 
		$\frac{1}{2}(q-\epsilon-1)(q-\epsilon)$ & 
		$q+\epsilon$\\
 &\\    
     % $\left(\begin{matrix}
              % \sigma^a & 0 \\ 0 & \sigma^{aq}
             %  \end{matrix} \right)$  & 
		$\frac{1}{2}(q+\epsilon-1)(q-\epsilon)$ & 
		$q-\epsilon$\\
 &\\
     \hline
    \end{tabular}
  \end{center}
\end{table}

\begin{landscape}
\begin{table}[t]
\small
\centering
\caption{Representatives of similarity classes in $\g^{\epsilon}_2(\mathfrak{o}_1)$ under $\G^{\epsilon}_2(\mathfrak{o}_2)$}
  \begin{center}
    \label{table22}
    \begin{tabular}{| c | c | c | c | c |}%{|*{4}{>{\centering\arraybackslash}p{3cm}|}}% {| c | c | c | c | }
      \hline
Type $A\in\mathbb{T}^{\epsilon}_2$& Parameter & Number of similarity & Isomorphism type & Index of $Z$ \\
   &  & classes, $n_A$ &$Z$ of $Z_{\G^{\epsilon}_2(\mathfrak{o}_1)}(A)$ & in $\G^{\epsilon}_2(\mathfrak{o}_1)$\\
      \hline
&&& &\\
      $\parbox{4.5cm}{$\{(t-\alpha)^{(1,1)}\}$}$  &
		$\parbox{4.2cm}{$\epsilon=1:\alpha\in\mathbb{F}_q$\\$\epsilon=-1: \alpha+\alpha^\circ=0$}$&
		$q$ & 
		$\G^{\epsilon}_2(\mathfrak{o}_1)$ &
		$1$\\
&&& &\\
      $\parbox{4.5cm}{$\{(t-\alpha)^{(2)}\}$}$  & 
		$\parbox{4.2cm}{$\epsilon=1:\alpha\in\mathbb{F}_q$\\$\epsilon=-1: \alpha+\alpha^\circ=0$}$&
		$q$ & 
		$\G^{\epsilon}_1(\mathfrak{o}_2)$ &
		$q^2-1$\\
&&& &\\
      $\parbox{4.5cm}{$\{(t-\alpha_1)^{(1)},(t-\alpha_2)^{(1)}\}$}$  & 
		$\parbox{4.2cm}{$\epsilon=1:\alpha_1\neq\alpha_2\in\mathbb{F}_q$\\$\epsilon=-1: \alpha_i+\alpha_i^\circ=0$\\$\alpha_1\neq\alpha_2$}$&
		$\frac{1}{2}q(q-1)$ & 
		$\G^{\epsilon}_1(\mathfrak{o}_1)\times\G^{\epsilon}_1(\mathfrak{o}_1)$ &
		$q(q+\epsilon)$\\
&&& &\\    
      $\parbox{4.5cm}{$\epsilon = 1: \{f^{(1)}\}$\\$\epsilon=-1:\{(t-\alpha_1)^{(1)},(t-\alpha_2)^{(1)}\}$}$  & 
		$\parbox{4.2cm}{$\epsilon=1:f$ irreducible quadratic\\$\epsilon=-1:\alpha_1=-\alpha_2^\circ$ distinct}$&
		$\frac{1}{2}q(q-1)$ & 
		$\mathbb{F}_{q^2}^*$ &
		$q(q-\epsilon)$\\
&& &&\\     
       \hline
    \end{tabular}
  \end{center}
%%\begin{tablenotes}
      %\small
      %\item If $\epsilon = 1$ then $\rho$ and $\sigma$ are primitive elements of $\mathbb{F}_{q}$ and $\mathbb{F}_{q^2}$ respectively with $\rho = \sigma^{q+1}$. If $\epsilon = -1$, then $\rho=\sigma^{q-1}$ where $\sigma$ is a primitive element of $\mathbb{F}_{q^2}$.%where $\rho$ and $\sigma$ are primitive elements of $\mathbb{F}_{q}$ and $\mathbb{F}_{q^2}$ respectively with $\rho = \sigma^{q+1}$.
   % \end{tablenotes}
\end{table}
\end{landscape}

\subsection{Representations of $\G^{\epsilon}_3(\mathfrak{o}_l)$}

We summarise what is already known about the number of complex irreducible representations of each degree of the groups $\G^{\epsilon}_3(\mathfrak{o}_l)$ for $\epsilon\in\{\pm 1\}$ and $l = 1, 2$. For $(\epsilon, n, l) = (1,3,1)$ see Steinberg \cite{Steinberg2}. For $(\epsilon, n, l) = (-1,3,1)$ see Ennola \cite{Ennola2}. The number of irreducible representations of the groups $\G^{\epsilon}_3(\mathfrak{o}_1)$ is given in Table \ref{table31}.

\begin{table}[h!]
\small
\centering
\caption{Representations of $\G^{\epsilon}_3(\mathfrak{o}_1)$}
  \begin{center}
%    \caption{$\GU_4(\mathfrak{o}_2)$}
    \label{table31}
    \begin{tabular}{| c | c |}%{|*{2}{>{\centering\arraybackslash}p{4cm}|}}% {| c | c | c | c | }
       \hline
      Number of irreducible representations & Degree \\
       \hline
 &\\
    %  $\left(\begin{matrix}
       %        \rho^a & 0 & 0 \\ 0 & \rho^a & 0 \\0 & 0 & \rho^a
          %     \end{matrix} \right)$  &
		$q-\epsilon$ & 
		$1$\\
 &\\
  %    $\left(\begin{matrix}
     %          \rho^a & 0 & 0\\ 1 & \rho^a & 0 \\0 & 0 & \rho^a
        %       \end{matrix} \right)$  & 
		$q-\epsilon$ & 
		$q(q+\epsilon)$\\
 &\\
  %    $\left(\begin{matrix}
     %          \rho^a & 0 & 0  \\ 1 & \rho^a & 0 \\0 & 1 & \rho^a
        %       \end{matrix} \right)$  & 
		$q-\epsilon$ & 
		$q^3$\\
 &\\     
  %    $\left(\begin{matrix}
     %          \rho^a & 0 & 0 \\ 0 & \rho^a & 0 \\0 & 0 & \rho^b
        %       \end{matrix} \right)$  & 
		$(q-\epsilon-1)(q-\epsilon)$ & 
		$q^2+\epsilon q+1$\\
 &\\  
  %    $\left(\begin{matrix}
     %          \rho^a & 0 & 0 \\ 1 & \rho^a & 0 \\0 & 0 & \rho^b
        %       \end{matrix} \right)$  & 
		$(q-\epsilon-1)(q-\epsilon)$ & 
		$q(q^2+\epsilon q+1)$\\
 &\\    
  %    $\left(\begin{matrix}
     %          \rho^a & 0 & 0 \\ 0 & \rho^b & 0 \\0 & 0 & \rho^c
        %       \end{matrix} \right)$  & 
		$\frac{1}{6}(q-\epsilon-2)(q-\epsilon-1)(q-\epsilon)$ & 
		$(q+\epsilon)(q^2+\epsilon q+ 1)$\\
 &\\ 
  %    $\left(\begin{matrix}
     %          \rho^a & 0 & 0 \\ 0 & \sigma^a & 0 \\0 & 0 & \sigma^{aq}
        %       \end{matrix} \right)$  & 
		$\frac{1}{2}(q+\epsilon-1)(q-\epsilon)^2$ & 
		$q^3-\epsilon$\\
 &\\
  %    $\left(\begin{matrix}
     %          \tau^a & 0 & 0 \\ 0 & \tau^{bq} & 0 \\0 & 0 & \tau^{bq^2}
        %       \end{matrix} \right)$  & 
		$\frac{1}{3}q(q^2-1)$ &
		$(q^2-1)(q-\epsilon)$\\
  &\\    
       \hline
    \end{tabular}
  \end{center}
\end{table}
  
For $l\geq 2$ we define $\G^{\epsilon}_{(l,1)}$ to be the group $H^{\epsilon}\rtimes D_l^{\epsilon}$, where
\begin{eqnarray*}
H^{1}&:=&\left\{\begin{pmatrix}1 & \alpha & \gamma \\ 0 & 1 & \beta \\ 0 & 0 & 1\end{pmatrix}:\alpha,\beta,\gamma\in\mathbb{F}_{q}\right\};\\
%D_l^{1}&:=&\left\{\begin{pmatrix}a & 0 & 0\\0 & b & 0\\0 & 0 & a\end{pmatrix}: a\in \GL_1(\mathfrak{o}_{l-1}), b\in \GL_1(\mathfrak{o}_{1})\right\};\\
H^{-1}&:=&\left\{\begin{pmatrix}1 & \alpha & \gamma \\ 0 & 1 & \bar{\alpha} \\ 0 & 0 & 1\end{pmatrix}:\alpha,\gamma\in\mathbb{F}_{q^2}, \alpha\bar{\alpha} = \gamma+\bar{\gamma}\right\};\\
%D_l^{-1}&:=&\left\{\begin{pmatrix}a & 0 & 0\\0 & b & 0\\0 & 0 & a\end{pmatrix}: a\in \GU_1(\mathfrak{o}_{l-1}), b\in \GU_1(\mathfrak{o}_{1})\right\}.
D_l^{\epsilon}&:=&\left\{\begin{pmatrix}a & 0 & 0\\0 & b & 0\\0 & 0 & a\end{pmatrix}: a\in \G^{\epsilon}_1(\mathfrak{o}_{l-1}), b\in \G^{\epsilon}_1(\mathfrak{o}_{1})\right\}.
\end{eqnarray*}

Before describing the irreducible representations of the groups $\G^{\epsilon}_3(\mathfrak{o}_2)$ we first describe the irreducible representations of the groups $\G^{\epsilon}_{(l,1)}$.

\begin{proposition}
The complex irreducible representations of the groups $\G^{\epsilon}_{(l,1)}$ for $l\geq 2$ are given in the Table 4. Hence, its zeta function is given by
$$
\zeta_{\G^\epsilon_{(l,1)}}(s)=q^{l-2}(q-\epsilon)((q-\epsilon)+(q+\epsilon)(q-\epsilon)^{-s}+(q-1)(q-\epsilon)q^{-s}).
$$
\end{proposition}

\begin{proof}
%Follows easily from the proof of Proposition 6.9 in \cite{AKOV3}.
This follows by adapting the proof of \cite[Proposition 6.9]{AKOV3} or \cite[Theorem 4.1]{Onn}. We will provide a sketch proof for completion. 
The group $H^{\epsilon}$ has $q-1$ irreducible representations of degree $q$ that correspond to the non-trivial characters of the centre, and $q^2$ linear characters factoring through its abelianisation by its centre $Z^{\epsilon} = Z(H^{\epsilon})$. Write $Q^{\epsilon} := H^{\epsilon}/Z^{\epsilon}=\mathbb{F}_q\times\mathbb{F}_q$. For each of the $q-1$ non-trivial characters of the centre, $\chi$, there is a unique irreducible representation $\rho_{\chi, H^\epsilon}$ of $H^{\epsilon}$ of dimension $q$. The remaining representations of $H^\epsilon$ correspond to the trivial character of the centre and hence factor through the quotient $Q^{\epsilon}$. Each of the representations $\rho_{\chi, H^\epsilon}$ is stabilised by $H^{\epsilon}$. Let $T^\epsilon=Z^{\epsilon}D^{\epsilon}\cong\G^\epsilon(\mathfrak{o}_l)\times\G^\epsilon(\mathfrak{o}_1)$ and let $H^\epsilon_1$ be a maximal abelian subgroup of $H^{\epsilon}$. Then $\chi$ can be extended from $Z^{\epsilon}=T^{\epsilon}\cap H^\epsilon_1$ to $T^\epsilon H^\epsilon_{1}$. Inducing the extension from $T^\epsilon H^\epsilon_1$ to $\G^\epsilon_{(l,1)}$ gives a $q$-dimensional representation which must extend $\rho_{\chi, H^\epsilon}$. This yields $|\G^{\epsilon}_{(l,1)}/H^{\epsilon}| = |T^\epsilon/Z^{\epsilon}|=q^{l-2}(q-\epsilon)^2$ different extensions of $\rho_{\chi,H^\epsilon}$ and so $q^{l-2}(q-\epsilon)^2(q-1)$ representations of $\G^{\epsilon}_{(l,1)}$. 

The remaining irreducible characters of $\G^{\epsilon}_{(l,1)}$ factor through its quotient by $Z^\epsilon$. The case $\epsilon = 1$ is done in \cite{Onn}. If $\epsilon = -1$ identify $Q^{-1}$ and its dual $Q^{-1\vee}$ with the additive group $\mathbb{F}_{q^2}$. The action of $\mbox{diag}(a,b,a)\in D_l^{-1}$ on $Q^{-1\vee}$ is given by $\mathbb{F}_{q^2}\ni u\mapsto(a^{-1}bu)$. The orbits of $D_l^{-1}$ on $Q^{-1\vee}$ are:

\begin{table}[!]
\small
\centering
%\caption{Orbits}
  \begin{center}
%    \caption{$\GU_4(\mathfrak{o}_2)$}
    \label{tab:table1}
    \begin{tabular}{| c | c | c |}%{|*{3}{>{\centering\arraybackslash}p{3cm}|}}% {| c | c | c | c | }
      \hline
      Orbit & Parameter & Stabiliser in $D$ \\
      \hline
& &\\
		$[0,0]$ & 
		$-$ &
		$\GU_1(\mathfrak{o}_{l-1})\times\GU_1(\mathfrak{o}_1)$\\
& &\\
		$[s]$ & 
		$s\in\mathbb{F}_{q^2}/\GU_1(\mathbb{F}_q)$ &
		$\GU_1(\mathfrak{o}_{l-1})$\\
& &\\

      \hline
    \end{tabular}
  \end{center}
%\small
\end{table}

By Mackey's method for semi-direct products (see \cite{JPS}, Section 8.2) this yields $|\GU_1(\mathfrak{o}_{l-1})\times\GU_1(\mathfrak{o}_1)| = q^{l-2}(q+1)^2$ linear characters and $|\mathbb{F}_{q^2}^*/\GU_1(\mathfrak{o}_{l-1})||\GU_1(\mathfrak{o}_1)| = q^{l-2}(q^2-1)$ irreducible characters of degree $|\GU_1(\mathfrak{o}_1)| = (q+1)$ of $\G_{(l,1)}^{-1}$. Putting this all together yields the required result.
\end{proof}

\begin{table}[!h!]
\small
\centering
\caption{Representations of $\G^{\epsilon}_{(l,1)}$}
  \begin{center}
%    \caption{$\GU_4(\mathfrak{o}_2)$}
    \label{tab:table1}
    \begin{tabular}{| c | c |}%{|*{2}{>{\centering\arraybackslash}p{4cm}|}}% {| c | c | c | c | }
       \hline
      Number of irreducible representations & Degree \\
      \hline
 &\\
      %$\left(\begin{matrix}
               %\rho^a & 0 \\ 0 & \rho^a
               %\end{matrix} \right)$  &
		$q^{l-2}(q-\epsilon)^2$ & 
		$1$\\
& \\
      %$\left(\begin{matrix}
               %\rho^a & 0\\ 1 & \rho^a
              % \end{matrix} \right)$  & 
		$q^{l-2}(q^2-1)$ & 
		$q-\epsilon$\\
&\\
      %$\left(\begin{matrix}
               %\rho^a & 0 \\ 0 & \rho^b
               %\end{matrix} \right)$  & 
		$q^{l-2}(q-\epsilon)^2(q-1)$ & 
		$q$\\
 &\\     
      \hline
    \end{tabular}
  \end{center}
\end{table}

The irreducible representations of $\G^{\epsilon}_n(\mathfrak{o}_l)$ for $(\epsilon, n, l) = (1,3,2), (-1,3,2)$ can be found in \cite{AKOV3}. Table \ref{table32} is a complete and irredundant list of representatives of similarity class types of $\g^{\epsilon}_3(\mathfrak{o}_1)$ under the action by $\G^{\epsilon}_3(\mathfrak{o}_2)$. Using equation (\ref{eqnzeta}) one can obtain the representation zeta function of the groups $\G^{\epsilon}_3(\mathfrak{o}_2)$.

\begin{landscape}
\begin{table}[t]
\small
\centering
\caption{Representatives of similarity classes in $\g^{\epsilon}_3(\mathfrak{o}_1)$ under $\G^{\epsilon}_3(\mathfrak{o}_2)$}
  \begin{center}
%    \caption{$\GU_4(\mathfrak{o}_2)$}
    \label{table32}
    \begin{tabular}{| c | c | c | c | c|}%{|*{4}{>{\centering\arraybackslash}p{3cm}|}}% {| c | c | c | c | }
       \hline
Type $A\in\mathbb{T}^{\epsilon}_3$ & Parameter & Number of similarity & Isomorphism type $Z$ of& Index of $Z$ in \\
  &   & classes, $n_A$ & $Z_{\G^{\epsilon}_3(\mathfrak{o}_1)}(A)$ & $\G^{\epsilon}_3(\mathfrak{o}_1)$\\
       \hline
&& &&\\
      $\parbox{6cm}{$\{(t-\alpha)^{(1,1,1)}\}$}$  &
		$\parbox{5.7cm}{$\epsilon=1:\alpha\in\mathbb{F}_q$\\$\epsilon=-1: \alpha+\alpha^\circ=0$}$&
		$q$ & 
		$\G^{\epsilon}_3(\mathfrak{o}_1) $ & 
		$1$\\
&&& &\\
      $\parbox{6cm}{$\{(t-\alpha)^{(2,1)}\}$}$  &
		$\parbox{5.7cm}{$\epsilon=1:\alpha\in\mathbb{F}_q$\\$\epsilon=-1: \alpha+\alpha^\circ=0$}$&
		$q$ & 
		$\G^{\epsilon}_{(2,1)}$ & 
		$(q^3-\epsilon)(q+\epsilon)$\\
&&& &\\
      $\parbox{6cm}{$\{(t-\alpha)^{(3)}\}$}$  &
		$\parbox{5.7cm}{$\epsilon=1:\alpha\in\mathbb{F}_q$\\$\epsilon=-1: \alpha+\alpha^\circ=0$}$&
		$q$ & 
		$\G^{\epsilon}_1(\mathfrak{o}_3)$ & 
		$q(q^2-1)(q^3-\epsilon)$\\
&&& &\\     
      $\parbox{6cm}{$\{(t-\alpha_1)^{(1,1)},(t-\alpha_2)^{(1)}\}$}$  & 
		$\parbox{5.7cm}{$\epsilon=1:\alpha_1\neq\alpha_2\in\mathbb{F}_q$\\$\epsilon=-1: \alpha_i+\alpha_i^\circ=0,\,\alpha_1\neq\alpha_2$}$&
		$q(q-1)$ & 
		$\G^{\epsilon}_1(\mathfrak{o}_1)\times \G^{\epsilon}_2(\mathfrak{o}_1)$ & 
		$q^2(q+\epsilon)(q^3-\epsilon)$\\
&&& &\\   
       $\parbox{6cm}{$\{(t-\alpha_1)^{(1,1)},(t-\alpha_2)^{(1)}\}$}$  & 
		$\parbox{5.7cm}{$\epsilon=1:\alpha_1\neq\alpha_2\in\mathbb{F}_q$\\$\epsilon=-1: \alpha_i+\alpha_i^\circ=0,\,\alpha_1\neq\alpha_2$}$&
		$q(q-1)$ & 
		$\G^{\epsilon}_1(\mathfrak{o}_1)\times\G^{\epsilon}_1(\mathfrak{o}_2)$ & 
		$q^2(q+\epsilon)(q^3-\epsilon)$\\
&&& &\\    
      $\parbox{6cm}{$\{(t-\alpha_1)^{(1)},(t-\alpha_2)^{(1)},(t-\alpha_3)^{(1)}\}$}$  & 
		$\parbox{5.7cm}{$\epsilon=1:\alpha_i\in\mathbb{F}_q$ distinct\\$\epsilon=-1: \alpha_i+\alpha_i^\circ=0$ distinct}$&
		$\frac{1}{6}q(q-1)(q-2)$ & 
		$\G^{\epsilon}_1(\mathfrak{o}_1)\times\G^{\epsilon}_1(\mathfrak{o}_1)
								\times\G^{\epsilon}_1(\mathfrak{o}_1)$ & 
		$q^3(q+\epsilon)(q^2+\epsilon q+ 1)$\\
&&& &\\ 
      $\parbox{6cm}{$\epsilon = 1: \{(t-\alpha)^{(1)},f^{(1)}\}$\\$\epsilon=-1:\{(t-\alpha_1)^{(1)},(t-\alpha_2)^{(1)},(t-\alpha_3)^{(1)}\}$}$  & 
		$\parbox{5.7cm}{$\epsilon=1:\alpha\in\mathbb{F}_q,\, f$ irreducible quadratic\\$\epsilon=-1:\alpha_1+\alpha_1^{\circ} = 0,\, \alpha_2=-\alpha_3^\circ$ distinct}$&
		$\frac{1}{2}q^2(q-1)$ & 
		$\G^{\epsilon}_1(\mathfrak{o}_1)\times\mathbb{F}_{q^2}^*$ & 
		$q^3(q^3-\epsilon)$\\
&& &&\\
      $\parbox{6cm}{$\{f^{(1)}\}$}$  & 
		$\parbox{5.7cm}{$f$ irreducible cubic *}$&
		$\frac{1}{3}q(q^2-1)$ &
		$\G^{\epsilon}_1(\mathbb{F}_{q^3})$ & 
		$q^3(q^2-1)(q-\epsilon)$\\
&& &&\\     
       \hline
    \end{tabular}
  \end{center}
\begin{tablenotes}
      \small
      \item *If $\epsilon=-1$ we require that $f=t^3+\sum_{i=0}^{2}c_it^i\in\mathbb{F}_{q^2}[t]$ where $c_i^{\circ}=(-1)^{i+1}c_i$ for $0\leq i<3$.
    \end{tablenotes}
\end{table}
\end{landscape}
\clearpage

\subsection{Representations of $\G^{\epsilon}_4(\mathfrak{o}_l)$}

The irreducible representations of $\G^{\epsilon}_n(\mathfrak{o}_l)$ for $(\epsilon, n, l) = (1,4,1)$ can be found in Steinberg \cite{Steinberg2} and are displayed in Table \ref{table41}. The case $(\epsilon, n, l) = (-1,4,1)$ is described by Nozawa \cite{Nozawa}. 

\begin{table}[h!]
\small
\centering
\caption{Representations of $\G^{\epsilon}_4(\mathfrak{o}_1)$}
  \begin{center}
%    \caption{$\GU_4(\mathfrak{o}_2)$}
    \label{table41}
    \begin{tabular}{| c | c |}%{|*{2}{>{\centering\arraybackslash}p{6cm}|}}% {| c | c | c | c | }
       \hline
      Number of irreducible representations  & Degree \\
       \hline
 &\\
  %    $\left(\begin{matrix}
     %          \rho^a & 0 & 0 & 0 \\ 0 & \rho^a & 0 & 0 \\0 & 0 & \rho^a & 0 \\ 0 & 0 & 0 & \rho^a 
        %       \end{matrix} \right)$  &
		$q-\epsilon$ & 
		$1$\\
 &\\
  %    $\left(\begin{matrix}
     %          \rho^a & 0 & 0 & 0 \\ 1 & \rho^a & 0 & 0 \\0 & 0 & \rho^a & 0 \\ 0 & 0 & 0 & \rho^a 
        %       \end{matrix} \right)$  & 
		$q-\epsilon$ & 
		$q(q^2+\epsilon q+1)$\\
 &\\
  %    $\left(\begin{matrix}
     %          \rho^a & 0 & 0 & 0 \\ 1 & \rho^a & 0 & 0 \\0 & 0 & \rho^a & 0 \\ 0 & 0 & 1 & \rho^a 
        %       \end{matrix} \right)$  & 
		$q-\epsilon$ & 
		$q^2(q^2+1)$\\
 &\\     
  %    $\left(\begin{matrix}
     %          \rho^a & 0 & 0 & 0 \\ 1 & \rho^a & 0 & 0 \\0 &1 & \rho^a & 0 \\ 0 & 0 & 0 & \rho^a 
        %       \end{matrix} \right)$  & 
		$q-\epsilon$ & 
		$q^3(q^2+\epsilon q+1)$\\
 &\\   
  %    $\left(\begin{matrix}
     %          \rho^a & 0 & 0 & 0 \\ 1 & \rho^a & 0 & 0 \\0 & 1 & \rho^a & 0 \\ 0 & 0 & 1 & \rho^a 
        %       \end{matrix} \right)$  & 
		$q-\epsilon$ & 
		$q^6$\\
 &\\    
  %    $\left(\begin{matrix}
     %          \rho^a & 0 & 0 & 0 \\ 0 & \rho^a & 0 & 0 \\0 & 0 & \rho^a & 0 \\ 0 & 0 & 0 & \rho^b 
        %       \end{matrix} \right)$  & 
		$(q-\epsilon-1)(q-\epsilon)$ & 
		$(q+\epsilon)(q^3+\epsilon)$\\
 &\\ 
  %    $\left(\begin{matrix}
     %          \rho^a & 0 & 0 & 0 \\ 1 & \rho^a & 0 & 0 \\0 & 0 & \rho^a & 0 \\ 0 & 0 & 0 & \rho^b 
        %       \end{matrix} \right)$  & 
		$(q-\epsilon-1)(q-\epsilon)$ & 
		$q(q^2+1)(q+\epsilon)^2$\\
 &\\
  %    $\left(\begin{matrix}
     %          \rho^a & 0 & 0 & 0 \\ 1 & \rho^a & 0 & 0 \\0 & 1 & \rho^a & 0 \\ 0 & 0 & 0 & \rho^b
        %       \end{matrix} \right)$  & 
		$(q-\epsilon-1)(q-\epsilon)$ &
		$q^3(q^3+\epsilon)(q+\epsilon)$\\
 &\\     
  %    $\left(\begin{matrix}
     %          \rho^a & 0 & 0 & 0 \\ 0 & \rho^a & 0 & 0 \\0 & 0 & \rho^b & 0 \\ 0 & 0 & 0 & \rho^b 
        %       \end{matrix} \right)$  & 
		$\frac{1}{2}(q-\epsilon-1)(q-\epsilon)$ & 
		$(q^2+1)(q^2+\epsilon q+1)$\\
  &\\  
   %   $\left(\begin{matrix}
      %         \rho^a & 0 & 0 & 0 \\ 1 & \rho^a & 0 & 0 \\0 & 0 & \rho^b & 0 \\ 0 & 0 & 0 & \rho^b 
         %      \end{matrix} \right)$  & 
		$(q-\epsilon-1)(q-\epsilon)$ & 
		$q(q^2+1)(q^2+\epsilon q+1)$\\
 &\\   
  %    $\left(\begin{matrix}
        %       \rho^a & 0 & 0 & 0 \\ 1 & \rho^a & 0 & 0 \\0 & 0 & \rho^b & 0 \\ 0 & 0 & 1 & \rho^b
           %    \end{matrix} \right)$  & 
		$\frac{1}{2}(q-\epsilon-1)(q-\epsilon)$ & 
		$q^2(q^2+1)(q^2+\epsilon q+1)$\\
 &\\      
  %    $\left(\begin{matrix}
     %          \rho^a & 0 & 0 & 0 \\ 0 & \rho^a & 0 & 0 \\0 & 0 & \rho^b & 0 \\ 0 & 0 & 0 & \rho^c
        %       \end{matrix} \right)$  & 
		$\frac{1}{2}(q-\epsilon-2)(q-\epsilon-1)(q-\epsilon)$ &
		$(q+\epsilon)(q^2+1)(q^2+\epsilon q+1)$\\
  &\\     
   %   $\left(\begin{matrix}
      %        \rho^a & 0 & 0 & 0 \\ 1 & \rho^a & 0 & 0 \\0 & 0 & \rho^b & 0 \\ 0 & 0 & 0 & \rho^c
         %      \end{matrix} \right)$  & 
		$\frac{1}{2}(q-\epsilon-2)(q-\epsilon-1)(q-\epsilon)$ & 
		$q(q+\epsilon)(q^2+1)(q^2+\epsilon q+1)$\\
 &\\     
  %    $\left(\begin{matrix}
     %         \rho^a & 0 & 0 & 0 \\ 0 & \rho^b & 0 & 0 \\0 & 0 & \rho^c & 0 \\ 0 & 0 & 0 & \rho^d 
        %       \end{matrix} \right)$  & 
		$\frac{1}{24}(q-\epsilon-3)(q-\epsilon-2)(q-\epsilon-1)(q-\epsilon)$ &
		$(q^2+1)(q+\epsilon)^2(q^2+\epsilon q+ 1)$\\
 &\\     
  %    $\left(\begin{matrix}
     %          \rho^a & 0 & 0 & 0 \\ 0 & \rho^a & 0 & 0 \\0 & 0 & \sigma^b & 0 \\ 0 & 0 & 0 & \sigma^{bq} 
        %       \end{matrix} \right)$  & 
		$\frac{1}{2}(q+\epsilon-1)(q-\epsilon)^2$ &  
		$(q-\epsilon)(q^2+1)(q^2+\epsilon q+1)$\\
 &\\     
  %    $\left(\begin{matrix}
     %          \rho^a & 0 & 0 & 0 \\ 1 & \rho^a & 0 & 0 \\0 & 0 & \sigma^b & 0 \\ 0 & 0 & 0 & \sigma^{bq} 
        %       \end{matrix} \right)$  & 
		$\frac{1}{2}(q+\epsilon-1)(q-\epsilon)^2$ & 
		$q(q-\epsilon)(q^2+1)(q^2+\epsilon q+1)$\\
  &\\     
   %   $\left(\begin{matrix}
      %          \rho^a & 0 & 0 & 0 \\ 0 & \rho^b & 0 & 0 \\0 & 0 & \sigma^c & 0 \\ 0 & 0 & 0 & \sigma^{cq} 
         %      \end{matrix} \right)$  & 
		$\frac{1}{4}q(q-2)(q-\epsilon)^2$ &
		$(q^4-1)(q^2+\epsilon q+1)$\\
 &\\      
  %    $\left(\begin{matrix}
     %          \sigma^a & 0 & 0 & 0 \\ 0 & \sigma^{aq} & 0 & 0 \\0 & 0 & \sigma^a & 0 \\ 0 & 0 & 0 & \sigma^{aq} 
        %       \end{matrix} \right)$  & 
		$\frac{1}{2}(q+\epsilon-1)(q-\epsilon)$ &  
		$q^2(q-\epsilon)^2(q^2+\epsilon q + 1)$\\
 &\\      
  %    $\left(\begin{matrix}
     %          \sigma^a & 0 & 0 & 0 \\ 0 & \sigma^{aq} & 0 & 0 \\1 & 0 & \sigma^a & 0 \\ 0 & 1 & 0 & \sigma^{aq} 
        %       \end{matrix} \right)$  & 
		$\frac{1}{2}(q+\epsilon-1)(q-\epsilon)$ & 
		$(q-\epsilon)^2(q^2+\epsilon q + 1)$\\
 &\\    
  %    $\left(\begin{matrix}
     %          \sigma^a & 0 & 0 & 0 \\ 0 & \sigma^{aq} & 0 & 0 \\0 & 0 & \sigma^b & 0 \\ 0 & 0 & 0 & \sigma^{bq} 
        %       \end{matrix} \right)$  & 
		$\frac{1}{8}(q^2-q-2)(q^2-q-2+2\epsilon)$ & 
		$(q-\epsilon)^2(q^2+1)(q^2+\epsilon q+1)$\\
 &\\ 
  %    $\left(\begin{matrix}
     %          \rho^a & 0 & 0 & 0 \\ 0 & \tau^b & 0 & 0 \\0 & 0 & \tau^{bq} & 0 \\ 0 & 0 & 0 & \tau^{bq^2}
        %       \end{matrix} \right)$  & 
		$\frac{1}{3}q(q^2-1)(q-\epsilon)$ &
		$(q^4-1)(q^2-1)$\\
 &\\     
  %    $\left(\begin{matrix}
     %          \omega^a & 0 & 0 & 0 \\ 0 & \omega^{aq} & 0 & 0 \\0 & 0 & \omega^{aq^2} & 0 \\ 0 & 0 & 0 & \omega^{aq^3} 
        %       \end{matrix} \right)$  & 
		$\frac{1}{4}q^2(q^2-1)$ & 
		$(q-\epsilon)^2(q^2-1)(q^2+\epsilon q+1)$\\
&\\
      \hline
    \end{tabular}
  \end{center}
\end{table}

\clearpage

We define $\G^{\epsilon}_{(2,1,1)}$ to be the group $E^{\epsilon}\rtimes M^{\epsilon}$ where
\begin{eqnarray*}
E^{1}&:=&\left\{\begin{pmatrix}1 & \alpha & \beta & \gamma \\ 0 & 1 & 0 & \delta \\ 0 & 0 & 1 &\eta\\ 0 & 0 & 0 & 1\end{pmatrix}:\alpha,\beta,\gamma, \delta,\eta\in\mathbb{F}_{q}\right\};\\
%M^{1}&:=&\left\{\begin{pmatrix}a & 0 & 0 & 0\\0 & w & z & 0\\0 & y & x & 0\\ 0 & 0 & 0 & a\end{pmatrix}: a\in \GL_1(\mathbb{F}_q), \begin{pmatrix}w & z\\y & x\end{pmatrix}\in\GL_2(\mathbb{F}_q)\right\};\\
E^{-1}&:=&\left\{\begin{pmatrix}1 & \alpha & \beta & \gamma \\ 0 & 1 & 0 & \bar{\alpha} \\ 0 & 0 & 1 & \bar{\beta} \\ 0 & 0 & 0 & 1\end{pmatrix}:\alpha,\beta,\gamma\in\mathbb{F}_{q^2}; \alpha\bar{\alpha}+\beta\bar{\beta}=\gamma+\bar{\gamma}\right\};\\
%M^{-1}&:=&\left\{\begin{pmatrix}a & 0 & 0 & 0\\0 & w & z & 0\\0 & y & x & 0\\ 0 & 0 & 0 & a\end{pmatrix}: a\in \GU_1(\mathbb{F}_q), \begin{pmatrix}w & z\\y & x\end{pmatrix}\in\GU_2(\mathbb{F}_q)\right\}.
M^{\epsilon}&:=&\left\{\begin{pmatrix}a & 0 & 0 & 0\\0 & w & z & 0\\0 & y & x & 0\\ 0 & 0 & 0 & a\end{pmatrix}: a\in \G^{\epsilon}_1(\mathfrak{o}_1), \begin{pmatrix}w & z\\y & x\end{pmatrix}\in\G^{\epsilon}_2(\mathfrak{o}_1)\right\}.
\end{eqnarray*}

Before describing the irreducible representations of $\G^{\epsilon}_n(\mathfrak{o}_l)$ for $(\epsilon, n, l) = (1,4,2), (-1,4,2)$ we must first describe the irreducible representations of the groups $\G^{\epsilon}_{(2,1,1)}$. 

\begin{proposition}
%The irreducible representations of the groups $\G^{\epsilon}_{(2,1,1)}$ can be obtained from the information in Table $7$. Hence,
The representation zeta function of the group $\G^{\epsilon}_{(2,1,1)}$ is given by
$$
\zeta_{\G^\epsilon_{(2,1,1)}}(s)=(q-1)(q-\epsilon)\zeta_{\G^\epsilon_2(\mathfrak{o}_1)}(s)q^{-2s} + \zeta_{\K^\epsilon}(s)
$$
where
%%\begin{eqnarray*}
$$
\K^{\epsilon} := E^{\epsilon}/Z(E^{\epsilon})\rtimes M^{\epsilon}
$$ %&=&(\mathbb{F}_{q^2}\times\mathbb{F}_{q^2})\rtimes(\G^{\epsilon}_1(\mathbb{F}_q)\times\G^{\epsilon}_2(\mathbb{F}_q))
%%\end{eqnarray*}
and
%$$
%\zeta_{\L^{1}}(s)=2(q-1)^2(q^2-1)^{-s}+(q-1)(q+2)((q^2-1)(q-1))^{-s}+(q-1)^3(q(q^2-1))^{-s}
%$$
%and
%$$
%\zeta_{\L^{-1}}(s)=(q^2-1)(q+1)(q(q^2-1))^{-s}+q(q^2-1)(q-1)(q+1)^{-2s}.
%$$
\begin{eqnarray*}
\zeta_{\K^1}(s)&=&(q-1)\zeta_{\GL_2(\mathfrak{o}_1)}(s)+2(q-1)^2(q^2-1)^{-s}+(q-1)(q+2)((q^2-1)(q-1))^{-s}+\\&&(q-1)^3(q(q^2-1))^{-s};\\
\zeta_{\K^{-1}}(s)&=&(q+1)\zeta_{\GU_2(\mathfrak{o}_1)}(s)+(q^2-1)(q+1)(q(q^2-1))^{-s}+q(q^2-1)(q-1)(q+1)^{-2s}.
\end{eqnarray*}

\end{proposition}

%\begin{table}[h!]
%\small
%\centering
%\caption{$\G^{\epsilon}_{(2,1,1)}$}
%  \begin{center}
%%    \caption{$\GU_4(\mathfrak{o}_2)$}
%    \label{tab:table1}
%    \begin{tabular}{| c | c | c |}%{|*{3}{>{\centering\arraybackslash}p{3cm}|}}% {| c | c | c | c | }
%      \hline
%Number of similarity & Isomorphism type & Index of $Z$ \\
%classes &$Z$ of $Z_{\G^{\epsilon}_4(\mathfrak{o}_1)}$ & in $\G^{\epsilon}_4(\mathfrak{o}_1)$\\
%      \hline
%& &\\
%		$(q-1)(q-\epsilon)$ & 
%		$\G^{\epsilon}_2(\mathfrak{o}_1)$ &
%		$q^2$\\
%& &\\
%		$q-\epsilon$ & 
%		$\G^{\epsilon}_2(\mathfrak{o}_1)$ &
%		$1$\\
%& &\\
%		$1$ & 
%		$\K^{\epsilon}$ &
%		$1$\\
%& &\\  
%      \hline
%    \end{tabular}
%  \end{center}
%\end{table}

\begin{proof}
The case $\epsilon = 1$ can be found in \cite{Pooja}. We proceed with the proof for $\epsilon = -1$ which is similar to the proof for $\epsilon = 1$. 
%Consider the map $G^{\epsilon}_{(2,1,1)}\rightarrow G^{\epsilon}_1\times G^{\epsilon}_{(1,1)}$ given by
%$$
%g\mapsto(g_{11}\mod\mathfrak{p}, \begin{pmatrix} g_{22} & g_{23} \\ g_{32} & g_{33}\end{pmatrix}).
%$$
For simplicity write $E:=E^{-1}$ and $M:=M^{-1}$ and let $H = E\rtimes M$. %, where
%$$
%E:=\left\{\begin{pmatrix}1 & \alpha & \beta & \gamma \\ 0 & 1 & 0 & \bar{\alpha} \\ 0 & 0 & 1 & \bar{\beta} \\ 0 & 0 & 0 & 1\end{pmatrix}:\alpha,\beta,\gamma\in\mathbb{F}_{q^2}, \alpha\bar{\alpha}+\beta\bar{\beta}=\gamma+\bar{\gamma}\right\}
%$$
%and
%$$
%M:=\left\{\begin{pmatrix}a & 0 & 0 & 0\\0 & w & z & 0\\0 & y & x & 0\\ 0 & 0 & 0 & a\end{pmatrix}: a\in \GU_1(\mathbb{F}_q), \begin{pmatrix}w & z\\y & x\end{pmatrix}\in\GU_2(\mathbb{F}_q)\right\}.
%$$
The group $E\cong H_4(\mathbb{F}_q)$, the Heisenberg group of degree $4$, has $q-1$ irreducible representations of dimension $q^2$ which lie above the non-trivial linear representations of the centre $Z=Z(E)\cong\mathbb{F}_q$. The group $M$ acts trivially on $Z$ and hence stabilises all the $q^2$-dimensional irreducible characters of $E$. Inducing these representations to $H$ contributes $(q-1)\zeta_M(s)q^{-2s}=(q-1)(q+1)q^{-2s}\zeta_{\G^{\epsilon}_2(\mathfrak{o}_1)}$ to the zeta function of $H$. 

We now deal with the remaining representations. These correspond to representations of $E$ whose central representation is trivial and factor through $Q = E/Z$. Consider $Q\rtimes M$ and identify $Q$ and its dual $Q^{\vee}$ with the additive group $\mathbb{F}_{q^2}\times\mathbb{F}_{q^2}$. The action of $m\in M$ on $Q^{\vee}$ is given by $\mathbb{F}_{q^2}\times\mathbb{F}_{q^2}\ni(u,v)\mapsto(a(u\bar{x}+v\bar{y}), a(-u\bar{D}y+vx\bar{D}))$ where we write an element $m\in M$ as
$$
m = \begin{pmatrix} a & 0 & 0 & 0\\0 & x & y & 0\\0 & -\bar{y}D & \bar{x}D & 0\\ 0 & 0 & 0 & a\end{pmatrix}
$$
for some $a,x,y,D\in\mathbb{F}_{q^2}$ satisfying $a\bar{a} = 1$, $x\bar{x}+y\bar{y} = 1$ and $D\bar{D} = 1$. We now use Mackey's method for semi-direct products (see \cite{JPS}, Section 8.2). The orbits of $M$ on $Q^{\vee}$ are:

\begin{table}[ht!]
\small
\centering
%\caption{Orbits}
  \begin{center}
%    \caption{$\GU_4(\mathfrak{o}_2)$}
    \label{tab:table1}
    \begin{tabular}{| c | c | c |}%{|*{3}{>{\centering\arraybackslash}p{3cm}|}}% {| c | c | c | c | }
      \hline
      Orbit & Parameter & Stabiliser in $D$ \\
      \hline
& &\\
		$[0,0]$ & 
		$-$ &
		$\G^{\epsilon}_2(\mathfrak{o}_1)\times\G^{\epsilon}_1(\mathfrak{o}_1)$\\
& &\\
		$[s,0]$ & 
		$s\in\mathbb{F}_{q^2}/\GU_1(\mathbb{F}_q)$ &
		$\G^{\epsilon}_1(\mathfrak{o}_1)\times \G^{\epsilon}_1(\mathfrak{o}_1)$\\
& &\\
		$[s,1]$ & 
		$s\in\mathbb{F}_{q^2}/\GU_1(\mathbb{F}_q)$ &
		$T$\\
& &\\  
      \hline
    \end{tabular}
  \end{center}
%\small
\end{table}

where 
$$
T=\left\{\begin{pmatrix}x & y\\ y & x\end{pmatrix}:x,y\in\mathbb{F}_{q^2}, x\bar{x}+y\bar{y}=1\right\}.
$$
This completes the proof.
\end{proof}

The irreducible representations of $\G^{\epsilon}_n(\mathfrak{o}_l)$ for $(\epsilon, n, l) = (1,4,2)$ can be found in \cite{Pooja}. We have now found every irreducible representation for ($\epsilon, n, l) = (-1,4,2)$ and we summarise this in Table \ref{table42}. The number of complex irreducible representations of each degree of the groups $\G^{\epsilon}_4(\mathfrak{o}_2)$ can be obtained from the information in Table \ref{table42} using equation (\ref{eqnzeta}).

\begin{theorem}\label{main}
The zeta function of the group $\G^{\epsilon}_4(\mathfrak{o}_2)$ is given by
$$
\zeta_{\G^{\epsilon}_4(\mathfrak{o}_2)}(s) = \sum_{A\in\mathbb{T}^{\epsilon}_4}n_A\zeta_{Z_{\G^{\epsilon}_4(\mathfrak{o}_1)}(A)}(s)|\G^{\epsilon}_4(\mathfrak{o}_1):Z_{\G^{\epsilon}_4(\mathfrak{o}_1)}(A)|^{-s},
$$
where $\mathbb{T}^{\epsilon}_4$ denotes the set of types of similarity classes in $\g^{\epsilon}_4(\mathfrak{o}_1)$.
\end{theorem}
%\clearpage
%\begin{proof}
%The case $\epsilon = 1$ can be found in \cite{Pooja}. The case $\epsilon = -1$ follows from above.
%\end{proof}
\begin{landscape}
\begin{table}[t]
\small
\centering
\caption{Representatives of similarity classes in $\g^{\epsilon}_4(\mathfrak{o}_1)$ under $\G^{\epsilon}_4(\mathfrak{o}_2)$}
  \begin{center}
%    \caption{$\GU_4(\mathfrak{o}_2)$}
    \label{table42}
    \begin{tabular}{| c | c | c | c | c |}
%    \begin{longtable}{|*{4}{>{\centering\arraybackslash}p{3.75cm}|}}% {| c | c | c | c | }
      \hline
Type $A\in\mathbb{T}^{\epsilon}_4$& Parameter & Number of similarity & Isomorphism type & Index of $Z$ in $\G^{\epsilon}_4(\mathfrak{o}_1)$ \\
   &  & classes, $n_A$ &$Z$ of $Z_{\G^{\epsilon}_4(\mathfrak{o}_1)}(A)$ & \\
      \hline
&&& &\\
      $\parbox{5cm}{$\{(t-\alpha)^{(1,1,1,1)}\}$}$  &
		$\parbox{4.4cm}{$\epsilon=1:\alpha\in\mathbb{F}_q$\\$\epsilon=-1: \alpha+\alpha^\circ=0$}$&
		$q$ & 
		$\G^{\epsilon}_4(\mathfrak{o}_1) $ & 
		$1$\\
&&& &\\
      $\parbox{5cm}{$\{(t-\alpha)^{(2,1,1)}\}$}$  &
		$\parbox{4.4cm}{$\epsilon=1:\alpha\in\mathbb{F}_q$\\$\epsilon=-1: \alpha+\alpha^\circ=0$}$&
		$q$ & 
		$\G^{\epsilon}_{(2,1,1)}$ & 
		$(q^2+1)(q^3-\epsilon)(q+\epsilon)$\\
&&& &\\
      $\parbox{5cm}{$\{(t-\alpha)^{(2,2)}\}$}$  &
		$\parbox{4.4cm}{$\epsilon=1:\alpha\in\mathbb{F}_q$\\$\epsilon=-1: \alpha+\alpha^\circ=0$}$&
		$q$ & 
		$\G^{\epsilon}_2(\mathfrak{o}_2)$ & 
		$q(q^4-1)(q^3-\epsilon)$\\
&&& &\\     
      $\parbox{5cm}{$\{(t-\alpha)^{(3,1)}\}$}$  &
		$\parbox{4.4cm}{$\epsilon=1:\alpha\in\mathbb{F}_q$\\$\epsilon=-1: \alpha+\alpha^\circ=0$}$&
		$q$ & 
		$\G^{\epsilon}_{(3,1)}$ & 
		$q^2(q^4-1)(q^3-\epsilon)(q+\epsilon)$\\
&& & &\\  
      $\parbox{5cm}{$\{(t-\alpha)^{(4)}\}$}$  &
		$\parbox{4.4cm}{$\epsilon=1:\alpha\in\mathbb{F}_q$\\$\epsilon=-1: \alpha+\alpha^\circ=0$}$&
		$q$ & 
		$\G^{\epsilon}_1(\mathfrak{o}_4)$ & 
		$q^3(q^4-1)(q^3-\epsilon)(q^2-1)$\\
&&& &\\    
      $\parbox{5cm}{$\{(t-\alpha_1)^{(1,1,1)},(t-\alpha_2)^{(1)}\}$}$  & 
		$\parbox{4.4cm}{$\epsilon=1:\alpha_1\neq\alpha_2\in\mathbb{F}_q$\\$\epsilon=-1: \alpha_i+\alpha_i^\circ=0,/,\alpha_1\neq\alpha_2$}$&
		$q(q-1)$ & 
		$\G^{\epsilon}_3(\mathfrak{o}_1)\times\G^{\epsilon}_1(\mathfrak{o}_1)$ & 
		$q^3(q+\epsilon)(q^2+1)$\\
& &&&\\ 
      $\parbox{5cm}{$\{(t-\alpha_1)^{(2,1)},(t-\alpha_2)^{(1)}\}$}$  & 
		$\parbox{4.4cm}{$\epsilon=1:\alpha_1\neq\alpha_2\in\mathbb{F}_q$\\$\epsilon=-1: \alpha_i+\alpha_i^\circ=0,\,\alpha_1\neq\alpha_2$}$&
		$q(q-1)$ & 
		$\G^{\epsilon}_{(2,1)}\times\G^{\epsilon}_1(\mathfrak{o}_1)$ & 
		$q^3(q^2+1)(q+\epsilon)^2(q^3-\epsilon)$\\
&&& &\\
      $\parbox{5cm}{$\{(t-\alpha_1)^{(3)},(t-\alpha_2)^{(1)}\}$}$  & 
		$\parbox{4.4cm}{$\epsilon=1:\alpha_1\neq\alpha_2\in\mathbb{F}_q$\\$\epsilon=-1: \alpha_i+\alpha_i^\circ=0,\,\alpha_1\neq\alpha_2$}$&
		$q(q-1)$ &
		$\G^{\epsilon}_1(\mathfrak{o}_3)\times\G^{\epsilon}_1(\mathfrak{o}_1)$ & 
		$q^4(q^4-1)(q^3-\epsilon)(q+\epsilon)$\\
&&& &\\    
      $\parbox{5cm}{$\{(t-\alpha_1)^{(1,1)},(t-\alpha_2)^{(1,1)}\}$}$  & 
		$\parbox{4.4cm}{$\epsilon=1:\alpha_1\neq\alpha_2\in\mathbb{F}_q$\\$\epsilon=-1: \alpha_i+\alpha_i^\circ=0,\, \alpha_1\neq\alpha_2$}$&
		$\frac{1}{2}q(q-1)$ & 
		$\G^{\epsilon}_2(\mathfrak{o}_1)\times\G^{\epsilon}_2(\mathfrak{o}_1)$ & 
		$q^4(q^2+1)(q^2+\epsilon q+1)$\\
&&& &\\   
%\hline
    $\parbox{5cm}{$\{(t-\alpha_1)^{(2)},(t-\alpha_2)^{(1,1)}\}$}$  & 
		$\parbox{4.4cm}{$\epsilon=1:\alpha_1\neq\alpha_2\in\mathbb{F}_q$\\$\epsilon=-1: \alpha_i+\alpha_i^\circ=0,\, \alpha_1\neq\alpha_2$}$&
		$q(q-1)$ & 
		$\G^{\epsilon}_1(\mathfrak{0}_2)\times\G^{\epsilon}_2(\mathfrak{o}_1)$ & 
		$q^4(q^2+\epsilon q+1)(q^4-1)$\\
&&& &\\   
      $\parbox{5cm}{$\{(t-\alpha_1)^{(2)},(t-\alpha_2)^{(2)}\}$}$  & 
		$\parbox{4.4cm}{$\epsilon=1:\alpha_1\neq\alpha_2\in\mathbb{F}_q$\\$\epsilon=-1: \alpha_i+\alpha_i^\circ=0,\, \alpha_1\neq\alpha_2$}$&
		$\frac{1}{2}q(q-1)$ & 
		$\G^{\epsilon}_1(\mathfrak{o}_2)\times\G^{\epsilon}_1(\mathfrak{o}_2)$ & 
		$q^4(q+\epsilon)(q^4-1)(q^3-\epsilon)$\\
&& &&\\      
      $\parbox{5cm}{$\{(t-\alpha_1)^{(1,1)},(t-\alpha_2)^{(1)},(t-\alpha_3)^{(1)}\}$}$  & 
		$\parbox{4.4cm}{$\epsilon=1:\alpha_i\in\mathbb{F}_q$ distinct\\$\epsilon=-1: \alpha_i+\alpha_i^\circ=0$, distinct}$&
		$\frac{1}{2}q(q-1)(q-2)$ &
		$\G^{\epsilon}_2(\mathfrak{o}_1)\times\G^{\epsilon}_1(\mathfrak{o}_1)^2$&
							%\times\G^{\epsilon}_1(\mathfrak{o}_1)$ & 
		$q^5(q+\epsilon)(q^2+1)(q^2+\epsilon q+1)$\\
&&& &\\        
\hline
    \end{tabular}
  \end{center}
\end{table}
\end{landscape}

\begin{landscape}
\begin{table}[t]
\small
\centering
%\caption{$\G^{\epsilon}_4(\mathfrak{o}_2)$}
  \begin{center}
%    \caption{$\GU_4(\mathfrak{o}_2)$}
    \label{tab:table1}
    \begin{tabular}{| c | c | c | c | c|}
%    \begin{longtable}{|*{4}{>{\centering\arraybackslash}p{3.75cm}|}}% {| c | c | c | c | }
      \hline
Type $A\in\mathbb{T}^{\epsilon}_4$& Parameter & Number of similarity & Isomorphism type & Index of $Z$ \\
   &  & classes, $n_A$ &$Z$ of $Z_{\G^{\epsilon}_4(\mathfrak{o}_1)}(A)$ & in $\G^{\epsilon}_4(\mathfrak{o}_1)$\\
      \hline
      $\parbox{6.5cm}{$\{(t-\alpha_1)^{(2)},(t-\alpha_2)^{(1)},(t-\alpha_3)^{(1)}\}$}$  & 
		$\parbox{6.75cm}{$\epsilon=1:\alpha_i\in\mathbb{F}_q$ distinct\\$\epsilon=-1: \alpha_i+\alpha_i^\circ=0$ distinct}$&
		$\frac{1}{2}q(q-1)(q-2)$ & 
		$\G^{\epsilon}_1(\mathfrak{o}_1)^3$&%\times\G^{\epsilon}_1(\mathfrak{o}_1)
								%\times\G^{\epsilon}_1(\mathfrak{o}_1)$ & 
		$q^5(q^2+1)(q+\epsilon)^2(q^3-\epsilon)$\\
&&& &\\    
      $\parbox{6.5cm}{$\{(t-\alpha_i)^{(1)}\}_{i=1,2,3,4}$}$  &%\{(t-\alpha_1)^{(1)},(t-\alpha_2)^{(1)},(t-\alpha_3)^{(1)},(t-\alpha_4)^{(1)}\}$}$  & \{(t-\alpha_1)^{(1)},(t-\alpha_2)^{(1)},(t-\alpha_3)^{(1)},(t-\alpha_4)^{(1)}\}$}$  & 
		$\parbox{6.75cm}{$\epsilon=1:\alpha_i\in\mathbb{F}_q$ distinct\\$\epsilon=-1: \alpha_i+\alpha_i^\circ=0$ distinct}$&
		$\frac{1}{24}q(q-1)(q-2)(q-3)$ &
		$\G^{\epsilon}_1(\mathfrak{o}_1)^4$&%\times\G^{\epsilon}_1(\mathfrak{o}_1)\times
				%\G^{\epsilon}_1(\mathfrak{o}_1)\times\G^{\epsilon}_1(\mathfrak{o}_1)$ & 
		$\parbox{3cm}{$q^6(q^3+\epsilon q^2+q+\epsilon)\times$\\$(q+\epsilon)(q^2+\epsilon q+ 1)$}$\\
&&& &\\     
      $\parbox{6.5cm}{$\epsilon = 1: \{(t-\alpha)^{(1,1)},f^{(1)}\}$\\$\epsilon=-1:\{(t-\alpha_1)^{(1,1)},(t-\alpha_2)^{(1)},(t-\alpha_3)^{(1)}\}$}$  & 
		$\parbox{6.75cm}{$\epsilon=1:\alpha\in\mathbb{F}_q,\,f$ irreducible quadratic\\$\epsilon=-1:\alpha_1+\alpha_1^{\circ} = 0,\,\alpha_2=-\alpha_3^\circ$ distinct}$&
		$\frac{1}{2}q^2(q-1)$ & 
		$\G^{\epsilon}_2(\mathfrak{o}_1)\times\mathbb{F}_{q^2}^*$ & 
		$q^5(q^3-\epsilon)(q^4-1)$\\
&& &&\\     
       $\parbox{6.5cm}{$\epsilon = 1: \{(t-\alpha)^{(2)},f^{(1)}\}$\\$\epsilon=-1:\{(t-\alpha_1)^{(2)},(t-\alpha_2)^{(1)},(t-\alpha_3)^{(1)}\}$}$  & 
		$\parbox{6.75cm}{$\epsilon=1:\alpha\in\mathbb{F}_q,\,f$ irreducible quadratic\\$\epsilon=-1:\alpha_1+\alpha_1^{\circ} = 0,\,\alpha_2=-\alpha_3^\circ$ distinct}$&
		$\frac{1}{2}q^2(q-1)$ & 
		$\G^{\epsilon}_1(\mathfrak{o}_2)\times\mathbb{F}_{q^2}^*$ & 
		$q^5(q^3-\epsilon)(q^4-1)$\\
&&& &\\     
       $\parbox{6.5cm}{$\epsilon = 1: \{(t-\alpha_1)^{(1)},(t-\alpha_2)^{(1)},f^{(1)}\}$\\$\epsilon=-1:\{(t-\alpha_i)^{(1)}\}_{i=1,2,3,4}$}$  &%\{(t-\alpha_1)^{(1)},(t-\alpha_2)^{(1)},(t-\alpha_3)^{(1)},(t-\alpha_4)^{(1)}\}$}$  & 
		$\parbox{6.75cm}{$\epsilon=1:\alpha_i\in\mathbb{F}_q$ distinct, $f$ irreducible quadratic\\$\epsilon=-1:\alpha_i+\alpha_i^{\circ} = 0,\,i=1,2$ distinct\\$\alpha_3=-\alpha_4^\circ$ distinct}$&
		$\frac{1}{4}q^2(q-1)^2$ &
		$\G^{\epsilon}_1(\mathfrak{o}_1)^2%\times\G^{\epsilon}_1(\mathfrak{o}_1)
										\times\mathbb{F}_{q^2}^*$ &
		$q^6(q+\epsilon)(q^2+1)(q^3-\epsilon)$\\
&&& &\\     
       $\parbox{6.5cm}{$\epsilon = 1: \{f^{(1,1)}\}$\\$\epsilon=-1:\{(t-\alpha_1)^{(1,1)},(t-\alpha_2)^{(1,1)}\}$}$  & 
		$\parbox{6.75cm}{$\epsilon=1:f$ irreducible quadratic\\$\epsilon=-1:\alpha_1=-\alpha_2^\circ$ distinct}$&
		$\frac{1}{2}q(q-1)$ & 
		$\GL_2(\mathbb{F}_{q^2})$ & 
		$q^4(q-\epsilon)(q^3-\epsilon)$\\
&&& &\\      
      $\parbox{6.5cm}{$\epsilon = 1: \{f^{(2)}\}$\\$\epsilon=-1:\{(t-\alpha_1)^{(2)},(t-\alpha_2)^{(2)}\}$}$  & 
		$\parbox{6.75cm}{$\epsilon=1:f$ irreducible quadratic\\$\epsilon=-1:\alpha_1=-\alpha_2^\circ$ distinct}$&
		$\frac{1}{2}q(q-1)$ & 
		$\mathbb{F}_{q^2}\times\mathbb{F}_{q^2}^*$ & 
		$q^4(q^4-1)(q^3-\epsilon)(q-\epsilon)$\\
& &&&\\    
      $\parbox{6.5cm}{$\epsilon = 1: \{f^{(1)},g^{(1)}\}$\\$\epsilon=-1:\{(t-\alpha_i)^{(1)}\}_{i=1,2,3,4}$}$  &%\{(t-\alpha_1)^{(1)},(t-\alpha_2)^{(1)},(t-\alpha_3)^{(1)},(t-\alpha_4)^{(1)}\}$}$  &  
		$\parbox{6.75cm}{$\epsilon=1:f\neq g$ irreducible quadratics\\$\epsilon=-1:\alpha_1=-\alpha_2^\circ$ distinct, $\alpha_3=-\alpha_4^\circ$ distinct\\$\{\alpha_1,\alpha_2\}\neq\{\alpha_3,\alpha_4\}$}$&
		$\frac{1}{8}q(q-1)(q^2-q-2)$ & 
		$\mathbb{F}_{q^2}^*\times\mathbb{F}_{q^2}^*$ & 
		$q^6(q^2+1)(q^3-\epsilon)(q-\epsilon)$\\
&& &&\\ 
      $\parbox{6.5cm}{$\{(t-\alpha)^{(1)},f^{(1)}\}$ $f$ irreducible cubic}$  & 
		$\parbox{6.75cm}{$\epsilon=1:\alpha\in\mathbb{F}_q$\\$\epsilon=-1:\alpha+\alpha^\circ=0$ *}$&
		$\frac{1}{3}q^2(q^2-1)$ &
		$\G^{\epsilon}_1(\mathfrak{o}_1)\times\G^{\epsilon}_1(\mathbb{F}_{q^3})$ & 
		$q^6(q^4-1)(q^2-1)$\\
 &&& &\\    
      $\parbox{6.5cm}{$\{f^{(1)}\}$ $f$ irreducible quartic}$  & 
		$\parbox{6.75cm}{**}$&
		$\frac{1}{4}q^2(q^2-1)$ & 
		$\mathbb{F}_{q^4}^*$ & 
		$q^6(q-\epsilon)(q^2-1)(q^3-\epsilon)$\\
 &&& &\\  
      \hline
    \end{tabular}
 \end{center}
\begin{tablenotes}
      \small
      \item * If $\epsilon=-1$ we require that $f=t^3+\sum_{i=0}^{2}c_it^i\in\mathbb{F}_{q^2}[t]$ where $c_i^{\circ}=(-1)^{i+1}c_i$ for $0\leq i<3$.\\
** If $\epsilon=-1$ we require that $f=t^4+\sum_{i=0}^{3}c_it^i\in\mathbb{F}_{q^2}[t]$ where $c_i^{\circ}=(-1)^{i+1}c_i$ for $0\leq i<4$.\\
    \end{tablenotes}
\end{table}
\end{landscape}

\clearpage

\section{Acknowledgements}
I would like to thank Christopher Voll for introducing me to this problem and for his support, guidance and insight throughout this project.

\bibliographystyle{plain}
\bibliography{refs}

\end{document}